\documentclass[12pt,a4paper]{article}
\usepackage{amsmath,amssymb,amsbsy,amscd,amsthm}
\newtheorem{thm}{Theorem}[section]
\newtheorem{lem}[thm]{Lemma}
\newtheorem{cor}[thm]{Corollary}

\newcommand{\trd }{\mathop{\rm trans.deg}\nolimits}
\newcommand{\ch }{\mathop{\rm char}\nolimits}
\newcommand{\pl }{\mathop{\rm pl}\nolimits}
\newcommand{\id}{{\rm id}}
\newcommand{\lcs }{\mathop{\rm lc}_{\sigma }\nolimits}
\newcommand{\degs}{\deg_{\sigma }}
\newcommand{\zs}{\{ 0\} }
\newcommand{\sm}{\setminus}
\newcommand{\Q}{{\bf Q}}
\newcommand{\Z}{{\bf Z}}

\begin{document}
\title{A generalization of Nakai's theorem 
on locally finite iterative higher derivations}

\author{Shigeru Kuroda}

\date{}


\footnotetext{Partly supported by the Grant-in-Aid for 
Young Scientists (B) 24740022, 
The Ministry of Education, Culture, 
Sports, Science and Technology, Japan.}

\maketitle

\begin{abstract}
Let $k$ be a field of arbitrary characteristic. 
Nakai (1978) proved a structure theorem 
for $k$-domains admitting a nontrivial 
locally finite iterative higher derivation 
when $k$ is algebraically closed. 
In this paper, 
we generalize Nakai's theorem to cover the case 
where $k$ is not algebraically closed. 
As a consequence, 
we obtain a cancellation theorem of the following form: 
Let $A$ and $A'$ be finitely generated $k$-domains 
with $A[x]\simeq _kA'[x]$. 
If $A$ and $\bar{k}\otimes _kA$ are UFDs 
and $\trd _kA=2$, 
then we have $A\simeq _kA'$. 
This generalizes the cancellation theorem of 
Crachiola (2009). 
\end{abstract}

\section{Introduction}
\label{sect:intro}
\setcounter{equation}{0}

Let $A$ be a commutative ring with identity, 
and $A[x]$ the polynomial ring in one variable over $A$. 
A homomorphism $\sigma :A\to A[x]$ of rings 
is called an {\it exponential map} on $A$ 
if the following conditions hold for each $a\in A$, 
where $a_0,\ldots ,a_m\in A$ are such that 
$\sigma (a)=\sum _{i=0}^ma_ix^i$, 
and $y$ is a new variable: 

\smallskip 

\quad 
(E1) $a_0=a$. \qquad 
(E2) $\sum _{i=0}^m\sigma (a_i)y^i=\sum _{i=0}^ma_i(x+y)^i$ 
in $R[x,y]$. 

\smallskip 

\noindent 
For each exponential map $\sigma $ on $A$, 
a collection $(\delta _i)_{i=0}^{\infty }$ 
of endomorphisms of the additive group $A$ 
is defined by 
$\sigma (a)=\sum _{i\geq 0}\delta _i(a)x^i$ 
for each $a\in A$. 
The so obtained $(\delta _i)_{i=0}^{\infty }$  
is called 
a {\it locally finite iterative higher derivation} 
on $A$. 
The notions of exponential maps 
and locally finite iterative higher derivations 
are of great importance in Affine Algebraic Geometry, 
especially in the study of the Cancellation Problem. 
Since both notions are equivalent, 
we consider exponential maps.

For each $a\in A$, 
we define $\degs (a)$ 
and $\lcs(a)$ 
to be the degree and leading coefficient 
of $\sigma (a)$ as a polynomial in $x$ over $A$, 
respectively. 
By (E1), 
the {\it ring} 
$A^{\sigma }:=\{ a\in A\mid \sigma (a)=a\} $ 
{\it of} $\sigma $-{\it invariants} 
is equal to $\sigma ^{-1}(A)$. 
Assume that $\sigma $ is {\it nontrivial}, i.e., 
$A^{\sigma }\neq A$. 
Then, 
we have $\degs (a)\geq 1$ for each $a\in A\sm A^{\sigma }$. 
We call $s\in A$ a {\it local slice} of $\sigma $ 
if $\degs (s)$ is equal to the minimum among 
$\degs (a)$ for $a\in A\sm A^{\sigma }$. 
A local slice $s$ of $\sigma $ is called a {\it slice} 
of $\sigma $ if $\lcs (s)=1$. 
There always exists a local slice, 
but a slice does not exist in general. 
It is known that, 
if $s$ is a slice of $\sigma $, 
then $A$ is the polynomial ring in $s$ 
over $A^{\sigma }$ 
(cf.~Theorem~\ref{thm:slice}). 
Even if $\sigma $ has no slice, 
$A$ can be a polynomial ring 
in one variable over $A^{\sigma }$ 
in some special cases. 
For example, 
let $k$ be a field of arbitrary characteristic. 
Then, 
this is the case 
if $A$ is the polynomial ring in two variables 
over $k$ 
(cf.~\cite{Rentschler} when $\ch k=0$, 
\cite{Miyanishi nagoya} when $k$ is algebraically closed, 
and \cite{Kojima} for the general case).

Nakai~\cite[Thm.\ 1]{Nakai} proved the following theorem. 
Here, 
for a subring $R$ of $A$, 
we say that $\sigma $ is an exponential map 
{\it over} $R$ if $R$ is contained in $A^{\sigma }$.

\begin{thm}[Nakai]\label{thm:Nakai}
Let $A$ be a $k$-domain, 
and $\sigma $ a nontrivial exponential map on $A$ over $k$. 
Assume that $k$ is algebraically closed, 
$A^{\sigma }$ is a finitely generated PID over $k$, 
and every prime element of $A^{\sigma }$ is 
a prime element of $A$. 
Then, 
$A$ is the polynomial ring in one variable over $A^{\sigma }$. 
\end{thm}

The purpose of this paper is to 
generalize Theorem~\ref{thm:Nakai}, 
and derive some useful consequences. 
One of our results 
implies the following theorem.

\begin{thm}\label{thm:cancel}
Let $k$ be any field, 
and $A$ and $A'$ 
finitely generated $k$-domains 
with $A[x]\simeq _kA'[x]$. 
If $A$ and $\bar{k}\otimes _kA$ are UFDs 
and $\trd _kA=2$, 
then we have $A\simeq _kA'$. 
Here, 
$\bar{k}$ is an algebraic closure of $k$. 
\end{thm}

This theorem is a generalization of 
Crachiola~\cite[Cor.\ 3.2]{tony} 
which says that 
$A[x]\simeq _kA'[x]$ implies $A\simeq _kA'$ 
if $A$ and $A'$ are finitely generated UFDs over 
an algebraically closed field $k$ 
with $\trd _kA=\trd _kA'=2$. 
One benefit of this generalization is that 
Theorem~\ref{thm:cancel} covers the case 
where $A$ is the polynomial ring 
in two variables over an arbitrary field $k$.

Thanks are due to Prof.\ Hideo Kojima 
for informing him of Nakai's paper.

\section{Main results}
\setcounter{equation}{0}

Comparing the coefficients of $y^m$, 
we have $\sigma (a_m)=a_m$ in (E2). 
Hence, 
$\lcs (a)$ belongs to $A^{\sigma }$ 
for each $a\in A$. 
If $\sigma $ is a nontrivial exponential map, 
then the {\it plinth ideal}
$$
\pl (\sigma ):=\{ \lcs (s)\mid 
\text{$s\in A$ is a local slice of $\sigma $}\} 
\cup \zs 
$$
is an ideal of $A^{\sigma }$. 
Actually, 
if $s,s'\in A$ are local slices 
with $\lcs (s)+\lcs (s')\neq 0$, 
and $a\in A^{\sigma }$ is such that $a\lcs (s)\neq 0$, 
then $s+s'$ and $as$ are local slices 
with $\lcs (s+s')=\lcs (s)+\lcs (s')$ 
and $\lcs (as)=a\lcs (s)$. 
The notion of plinth ideal 
already appeared in Nakai~\cite{Nakai}, 
although not called by this name. 
A local slice $s$ of $\sigma $ 
is said to be {\it minimal} 
if there does not exist 
$a\in \pl (\sigma )$ such that 
$\lcs (s)A^{\sigma }\subsetneqq aA^{\sigma }$. 
If $s\in A$ satisfies 
$\pl (\sigma )=\lcs (s)A^{\sigma }$, 
then $s$ is a minimal local slice of $\sigma $.

Now, 
let $k$ be any field, 
$A$ a commutative $k$-algebra, 
and $\sigma $ 
a nontrivial exponential map on $A$ over $k$. 
Set $\bar{A}:=\bar{k}\otimes _kA$ 
and $\bar{\sigma }:=\id _{\bar{k}}\otimes \sigma $, 
where $\bar{k}$ is an algebraic closure of $k$. 
Then, 
$\bar{\sigma }$ is an exponential map 
on $\bar{A}$ over $\bar{k}$ with 
$\bar{A}^{\bar{\sigma }}=\bar{k}\otimes _kA^{\sigma }$. 
In this notation, 
we have the following theorem.

\begin{thm}\label{thm:main}
Let $s\in A$ be a minimal local slice of $\bar{\sigma }$. 
Assume that $\lcs (s)$ 
is a nonzero divisor of $\bar{A}$, and 
is written as $p_1\cdots p_l$ with $l\geq 0$, 
where $p_1,\ldots ,p_l\in \bar{A}^{\bar{\sigma }}$ 
are such that 
$\bar{A}^{\bar{\sigma }}/p_i\bar{A}^{\bar{\sigma }}=\bar{k}$ 
and $\bar{A}/p_i\bar{A}$ is a domain 
for $i=1,\ldots ,l$. 
Then, 
$A$ is the polynomial ring in $s$ over $A^{\sigma }$. 
\end{thm}

Observe that every local slice of $\bar{\sigma }$ 
is written as a $\bar{k}$-linear combination 
of local slices of $\sigma $ 
and elements of $A^{\sigma }$. 
Hence, 
we get 
$\pl (\bar{\sigma })=\bar{k}\otimes _k\pl (\sigma )$. 
Thus, 
if $s\in A$ satisfies 
$\pl (\sigma )=\lcs (s)A^{\sigma }$, 
then we have 
$\pl (\bar{\sigma })=\bar{k}\otimes _k\pl (\sigma )
=\lcs (s)\bar{A}^{\bar{\sigma }}$, 
and so $s$ is a minimal local slice 
of $\bar{\sigma }$. 
Therefore, 
if $A$ admits a nontrivial exponential map 
$\sigma $ over $k$ 
with the following three conditions, 
then $A$ is a polynomial ring 
in one variable over $A^{\sigma }$ 
by Theorem~\ref{thm:main}:

\smallskip

\noindent (N1) $\pl (\sigma )$ is a principal ideal 
of $A^{\sigma }$ 
generated by a nonzero divisor of $\bar{A}$.

\noindent (N2) 
$\bar{A}^{\bar{\sigma }}$ is a PID 
with $\trd _k\bar{A}^{\bar{\sigma }}=1$.

\noindent (N3) 
$\bar{A}/p\bar{A}$ is a domain 
for every prime element $p$ of 
$\bar{A}^{\bar{\sigma }}$.

\smallskip

\noindent 
Since a finitely generated PID over a field $k$ 
has transcendence degree one over $k$, 
we see that 
the assumption of Theorem~\ref{thm:Nakai} 
implies (N1), (N2) and (N3). 
Therefore, 
we obtain Theorem~\ref{thm:Nakai} 
from Theorem~\ref{thm:main}.

Next, 
assume that $A$ is a domain. 
It is well known that 
$A^{\sigma }=\sigma ^{-1}(A)$ 
is {\it factorially closed} in $A$, 
i.e., 
$ab\in A^{\sigma }$ implies $a,b\in A^{\sigma }$ 
for each $a,b\in A\sm \zs $, 
since $A$ is factorially closed in $A[x]$, 
and $\sigma $ is injective by (E1). 
This implies that $(A^{\sigma })^*=A^*$, 
and every irreducible element of $A^{\sigma }$ 
is an irreducible element of $A$. 
Note that, 
if $p\in A^{\sigma }$ is a prime element of $A$, 
then $p$ is a prime element of $A^{\sigma }$, 
since $pA^{\sigma }=pA\cap A^{\sigma }$ 
is a prime ideal. 
Hence, 
if $A$ is a UFD, 
then $A^{\sigma }$ is also a UFD, 
and every prime element of $A^{\sigma }$ 
is a prime element of $A$. 
We also note that, 
if $A$ contains a field $k$, 
then every exponential map $\sigma $ on $A$ 
is an exponential map over $k$, 
since $k\sm \zs \subset A^*\subset A^{\sigma }$.

The following corollary is also a consequence of 
Theorem~\ref{thm:main}.

\begin{cor}\label{cor:UFD}
Let $A$ be a UFD over a field $k$, 
and $\sigma $ a nontrivial exponential map on $A$. 
If $\trd _kA=2$ and $\bar{k}\otimes _kA$ is a UFD, 
then $A$ is a polynomial ring 
in one variable over $A^{\sigma }$. 
\end{cor}

To see this, 
we check that the assumption of 
Corollary~\ref{cor:UFD} 
implies (N1), (N2) and (N3). 
Since $A$ and $\bar{A}$ are UFDs, 
we know that 
$A^{\sigma }$ and $\bar{A}^{\bar{\sigma }}$ 
are UFDs, 
and (N3) is fulfilled. 
Recall that a UFD is a PID if 
every nonzero principal prime ideal is maximal. 
Since $\trd _kA=2$ by assumption, 
we have 
$\trd _kA^{\sigma }=\trd _k\bar{A}^{\bar{\sigma }}=1$ 
(cf.~Corollary~\ref{cor:local slice}), 
and so $A^{\sigma }$ 
and $\bar{A}^{\bar{\sigma }}$ are PIDs. 
Therefore, 
we get (N1) and (N2).

Crachiola~\cite[Theorem 3.1]{tony} showed 
Corollary~\ref{cor:UFD} when $k$ is algebraically closed. 
He derived from this result 
the cancellation theorem~\cite[Cor.\ 3.2]{tony} 
mentioned in Section~\ref{sect:intro} 
by making use of 
Crachiola--Makar-Limanov~\cite[Thm.\ 3.1]{CML} 
and Abhyankar-Eakin-Heinzer~\cite[Thm.~3.3]{AEH}. 
His argument in fact proved the following statement 
for an arbitrary field $k$ 
(see the proof of \cite[Cor.\ 3.2]{tony}).

\begin{lem}\label{lem:Crachiola}
For $i=1,2$, 
let $A_i$ be a finitely generated $k$-domain 
with $\trd _kA_i=2$ having 
the following property$:$ 
If $A_i$ admits 
a nontrivial exponential map $\sigma $ over $k$, 
then $A_i$ is a polynomial ring 
in one variable over $A_i^{\sigma }$. 
Then, 
it holds that 
$A_1[x]\simeq _kA_2[x]$ implies 
$A_1\simeq _kA_2$. 
\end{lem}

Under the assumption of Theorem~\ref{thm:cancel}, 
$A'[x]$ and $\bar{k}\otimes _kA'[x]$ are UFDs, 
and $\trd _kA'=2$. 
Since $A'$ and $\bar{k}\otimes _kA'$ are 
factorially closed in 
$A'[x]$ and $\bar{k}\otimes _kA'[x]$, 
respectively, 
it follows that 
$A'$ and $\bar{k}\otimes _kA'$ are also UFDs. 
Hence, 
using Lemma~\ref{lem:Crachiola}, 
we can derive Theorem~\ref{thm:cancel} 
from Corollary~\ref{cor:UFD}.

Let $A=R[x,y]$ be the polynomial ring 
in two variables over a domain $R$. 
Corollary~\ref{cor:UFD} also implies 
the result that, 
if $R$ is a field, 
and $\sigma $ is a nontrivial exponential map on $A$, 
then $A$ is a polynomial ring in one variable 
over $A^{\sigma }$ (cf.\ \S 1). 
The following theorem is a consequence of this result. 
We call $f\in R[x,y]$ a {\it coordinate} of $R[x,y]$ 
if there exists $g\in R[x,y]$ such that $R[x,y]=R[f,g]$. 
A domain $R$ is called 
an {\it HCF-ring} if, 
for any $a,b\in R$, 
there exists $c\in R$ such that $aR\cap bR=cR$. 
For example, 
UFDs are HCF-rings.

\begin{thm}\label{thm:HCF}
Let $R$ be an HCF-ring, 
$K$ the field of fractions of $R$, 
and $\sigma $ a nontrivial exponential map on $R[x,y]$ 
over $R$. 
Then, 
there exists $f\in R[x,y]$ 
such that $f$ is a coordinate of $K[x,y]$ 
and $R[x,y]^{\sigma }=R[f]$. 
\end{thm}

When $R$ contains $\Q $, 
this result is found in 
Freudenburg~\cite[Thms.\ 4.11 and 4.13]{Fbook}. 
The proof of Theorem~\ref{thm:HCF} is similar: 
$R[x,y]^{\sigma }$ is factorially closed in $R[x,y]$, 
and is of transcendence degree one over $R$ 
(cf.~Corollary~\ref{cor:local slice}). 
Since $R$ is an HCF-ring by assumption, 
this implies that 
$R[x,y]^{\sigma }=R[f]$ for some $f\in R[x,y]$ 
by Abhyankar-Eakin-Heinzer~\cite[Prop.\ 4.8]{AEH}. 
Let $\tilde{\sigma }$ be the extension of 
$\sigma $ to $K[x,y]$. 
Then, 
we have $K[x,y]^{\tilde{\sigma }}=K[f]$, 
and $K[x,y]=K[x,y]^{\tilde{\sigma }}[g]=K[f,g]$ 
for some $g\in K[x,y]$ 
by the result mentioned above.

\section{Proof of Theorem~\ref{thm:main}}
\label{sect:proof}
\setcounter{equation}{0}

First, 
we remark that, 
if $s$ is a minimal local slice 
with $\lcs(s)$ a nonzero divisor of $A^{\sigma }$, 
then the image of $s$ in $A/qA$ 
is not contained in 
the image of $A^{\sigma }$ 
for any $q\in A^{\sigma }\sm A^*$. 
In fact, 
if $s-b=qs'$ holds for some $b\in A^{\sigma }$ and $s'\in A$, 
then we have $\sigma (s)-b=q\sigma (s')$, 
and so $\lcs (s)=qc$ for some $c\in A$. 
Since $\lcs (s)$ is a nonzero divisor of $A^{\sigma }$, 
it follows that so is $q$. 
Since $\lcs (s')$ belongs to $A^{\sigma }$, 
we get $\degs (s)=\degs (s')$ and 
$\lcs (s)=q\lcs (s')$, 
contradicting the minimality of $s$.

Now, 
let us prove Theorem~\ref{thm:main}. 
By assumption, 
$s$ is a minimal local slice of $\bar{\sigma }$, 
and $a:=\lcs (s)=p_1\cdots p_l$ 
is a nonzero divisor of $\bar{A}$. 
Hence, 
$\bar{A}[a^{-1}]$ is the polynomial ring in $s$ 
over $\bar{A}^{\bar{\sigma }}[a^{-1}]$ 
(cf.~Corollary~\ref{cor:local slice}). 
Since $A^{\sigma }$ is contained in 
$\bar{A}^{\bar{\sigma }}[a^{-1}]$, 
it suffices to verify 
$A=A^{\sigma }[s]$. 
Since $A\supset A^{\sigma }[s]$, 
we prove that 
$\bar{k}\otimes _kA=\bar{k}\otimes _kA^{\sigma }[s]$, 
that is, 
$\bar{A}=\bar{A}^{\bar{\sigma }}[s]$. 
This is clear if $a=1$. 
So assume that $l\geq 1$. 
We remark that, 
if $c\in \bar{A}$ satisfies $p_1\cdots p_ic=d$ for some 
$1\leq i\leq l$ and $d\in \bar{A}^{\bar{\sigma }}$, 
then $c$ belongs to $\bar{A}^{\bar{\sigma }}$. 
In fact, 
since $p_1,\ldots ,p_i$ are elements of $\bar{A}^{\bar{\sigma }}$ 
by assumption, 
we have 
$p_1\cdots p_i\bar{\sigma }(c)=d$, 
and so 
$\bar{\sigma }(c)=(p_1\cdots p_i)^{-1}d=c$ 
in $\bar{A}[a^{-1}]$. 
Similarly, 
$p_i\bar{A}\cap \bar{A}^{\bar{\sigma }}
=p_i\bar{A}^{\bar{\sigma }}$ 
holds for each $i$. 
Now, 
take any $b\in \bar{A}$, 
and write $b=a^{-n}\sum _{i\geq 0}b_is^i$, 
where $n\geq 0$ and $b_i\in \bar{A}^{\bar{\sigma }}$ 
for each $i$. 
We may assume that $n$ is minimal among such expressions. 
To conclude $b\in \bar{A}^{\bar{\sigma }}[s]$, 
we show that $n=0$ by contradiction. 
Suppose that $n\geq 1$. 
Then, 
$\sum _{i\geq 0}b_is^i=a^nb$ belongs to $a\bar{A}$. 
Let $1\leq u\leq l+1$ 
be the maximal number satisfying 
$\{ b_i\mid i\geq 0\} \subset p_1\cdots p_{u-1}\bar{A}$. 
Then, 
by the minimality of $n$, 
we have $1\leq u\leq l$. 
Set 
$c_i:=(p_1\cdots p_{u-1})^{-1}b_i\in \bar{A}^{\bar{\sigma }}$ 
for each $i$. 
Then, 
$\sum _{i\geq 0}c_is^i$ belongs to $p_u\bar{A}$, 
but $c_{i_0}$ does not belong to $p_u\bar{A}$ 
for some $i_0$ by the maximality of $u$. 
By assumption, 
$\bar{A}/p_u\bar{A}$ is a domain. 
Moreover, 
the image of $\bar{A}^{\bar{\sigma }}$ 
in $\bar{A}/p_u\bar{A}$ 
is equal to $\bar{k}$, 
since 
$p_u\bar{A}\cap \bar{A}^{\bar{\sigma }}
=p_u\bar{A}^{\bar{\sigma }}$, 
and 
$\bar{A}^{\bar{\sigma }}/p_u\bar{A}^{\bar{\sigma }}=\bar{k}$ 
by assumption. 
Hence, 
the image of $s$ in $\bar{A}/p_u\bar{A}$ 
is algebraic over $\bar{k}$, 
and thus belongs to the image of $\bar{A}^{\bar{\sigma }}$ 
in $\bar{A}/p_u\bar{A}$. 
Since $\lcs (s)$ is a nonzero divisor of $\bar{A}$, 
this contradicts the minimality of $s$ by the first remark. 
Therefore, 
$b$ belongs to $\bar{A}^{\bar{\sigma }}[s]$, 
proving $\bar{A}=\bar{A}^{\bar{\sigma }}[s]$. 
This completes the proof of Theorem~\ref{thm:main}.

\section{Appendix}
\setcounter{equation}{0}

Let $A$ be any commutative ring, 
and $\sigma $ a nontrivial exponential map on $A$. 
In this appendix, 
we prove the following theorem 
for the lack of a suitable reference 
(cf.\ e.g.~\cite[Lem.\ 2.2]{CML} when $A$ is a domain).

\begin{thm}\label{thm:slice}
If $\sigma $ has a slice $s$, 
then $A$ is the polynomial ring in $s$ over $A^{\sigma }$. 
\end{thm}

First, 
note that {\rm (E1)} and {\rm (E2)} 
imply the following statements.

\begin{lem}\label{lem:iterative}
\noindent{\rm (i)} For each $0\leq i\leq m$, 
we have $\degs (a_i)\leq m-i$.

\noindent{\rm (ii)} 
Assume that $p:=\ch (A)$ is a prime number, 
and write $m=lp^e$, 
where $e\geq 0$ and $l\geq 1$ with $p\nmid l$. 
If $a_m\neq 0$, 
then we have $\degs (a_{(l-1)p^e})=p^e$.

\noindent{\rm (iii)} 
If $\sigma (a_i)=a_i$ for all $1\leq i\leq m$, 
then $f(x):=\sigma (a)-a
=\sum _{i=1}^ma_ix^i$ is {\rm additive}, 
i.e., 
satisfies $f(x+y)=f(x)+f(y)$. 
\end{lem}
\begin{proof}
Considering total degrees, 
(i) is clear from (E2). 
In the case (ii), 
we have $l\in A^*$ and 
$$
(x+y)^m=(x+y)^{lp^e}=(x^{p^e}+y^{p^e})^l
=y^{lp^e}+lx^{p^e}y^{(l-1)p^e}+\cdots +x^{lp^e}. 
$$
By (E2), 
this implies 
$\sigma (a_{(l-1)p^e})
=la_mx^{p^e}+(\text{terms of lower degree in $x$})$ 
with $la_m\neq 0$. 
In the case (iii), 
we have 
$$
\sum _{i=0}^m\sigma (a_i)y^i
=\sigma (a)+\sum _{i=1}^ma_iy^i
=\sum _{i=0}^ma_ix^i+\sum _{i=1}^ma_iy^i
=a_0+f(x)+f(y), 
$$
since $a_0=a$ by (E1). 
Hence, 
we know by (E2) that $f(x)$ is additive. 
\end{proof}

For each integer $n\geq 2$, 
let $d(n)$ be the greatest common divisor of 
the binomial coefficients 
$\binom{n}{i}$ for $1\leq i<n$. 
If $n=p^d$ for some prime number $p$ 
and $d\geq 1$, 
then we have $d(n)=p$, 
since $p^2\nmid \binom{p^d}{p^{d-1}}$. 
Otherwise, 
we have $d(n)=1$. 
Hence, 
the following lemma holds.

\begin{lem}\label{lem:binom}
For an integer $n\geq 2$ and $a\in A\sm \zs $, 
we have $a(x+y)^n=a(x^n+y^n)$ 
if and only if there exist a prime number $p$ 
and $d\geq 1$ such that $n=p^d$ 
and $\{ l\in \Z \mid la=0\} =p\Z $. 
\end{lem}

Let $s$ be a local slice of $\sigma $. 
Then, 
by the minimality of $n:=\degs (s)$, 
we see from Lemma~\ref{lem:iterative} (i) 
that the coefficient of $x^i$ in $\sigma (s)$ 
belongs to $A^{\sigma }$ for $i=1,\ldots ,n$. 
Hence, 
$\sigma (s)-s$ is additive 
by Lemma~\ref{lem:iterative} (iii). 
If furthermore $s$ is a slice of $\sigma $, 
then this implies either $n=1$, 
or $p:=\ch (A)$ is a prime number 
and $n=p^d$ for some $d\geq 1$ 
by Lemma~\ref{lem:binom} with $a=1$.

Now, let us prove Theorem~\ref{thm:slice}. 
First, 
we show that each $a\in A\sm \zs $ belongs to $A^{\sigma }[s]$ 
by induction on $m:=\degs (a)$. 
If $m=0$, 
then $a$ belongs to $A^{\sigma }$. 
Assume that $m>0$. 
We show that $n:=\degs (s)$ divides $m$. 
We may assume that $n\geq 2$. 
Then, 
$p:=\ch (A)$ is a prime number, 
and $n=p^d$ 
for some $d\geq 1$ 
as mentioned. 
Write $m=lp^e$ with $p\nmid l$ and $e\geq 0$. 
Then, 
$\degs (b)=p^e$ holds for some $b\in A$ 
by Lemma~\ref{lem:iterative} (ii). 
By the minimality of $n$, 
it follows that $n=p^d$ divides $p^e$, 
and hence divides $m$. 
Now, 
set $a':=\lcs (a)$ and $c:=a-a's^{m/n}$. 
Then, 
the degree $\degs (c)$ of 
$\sigma (c)=\sigma (a)-a'\sigma (s)^{m/n}$ 
is less than $m$. 
Hence, 
$c$ belongs to $A^{\sigma }[s]$ 
by induction assumption. 
Thus, 
$a$ belongs to $A^{\sigma }[s]$, 
since so does $a's^{m/n}$. 
Therefore, 
we have $A=A^{\sigma }[s]$. 
Next, 
suppose that $A^{\sigma }[s]$ is not the polynomial ring in 
$s$ over $A^{\sigma }$. 
Then, 
there exist $m\geq 1$ and $a_0,\ldots ,a_m\in A^{\sigma }$ 
with $a_m\neq 0$ such that $\sum _{i=0}^ma_is^i=0$. 
Since $\sum _{i=0}^ma_i\sigma (s)^i=0$, 
and $\sigma (s)$ is a monic polynomial 
of positive degree, 
we are led to a contradiction. 
This completes the proof of Theorem~\ref{thm:slice}.

The following corollary is a consequence of 
Theorem~\ref{thm:slice}, 
since $\sigma $ extends to a nontrivial 
exponential map $\tilde{\sigma }$ on $A[a^{-1}]$ 
with $A[a^{-1}]^{\tilde{\sigma }}=A^{\sigma }[a^{-1}]$ 
which has a slice $a^{-1}s$.

\begin{cor}\label{cor:local slice}
Let $s$ be a local slice of $\sigma $ 
such that $a:=\lcs (s)$ is a nonzero divisor of $A$. 
Then, 
$A[a^{-1}]$ is the polynomial ring in $s$ over 
$A^{\sigma }[a^{-1}]$. 
Hence, 
we have $\trd _{A^{\sigma }}A=1$ 
if $A$ is a domain. 
\end{cor}

\noindent
Department of Mathematics and Information Sciences \\ 
Tokyo Metropolitan University \\ 
1-1  Minami-Osawa, Hachioji, 
Tokyo 192-0397, Japan\\
kuroda@tmu.ac.jp

\end{document}